\documentclass[11pt]{article}

\usepackage[T2A]{fontenc}
\usepackage[cp1251]{inputenc}
\usepackage{amsthm}%%%%%%%% theorem definition
\usepackage{amsfonts}
\usepackage{amssymb}
\usepackage{amsmath}
\usepackage{cite}%%%%% citations

\makeatletter
\renewcommand{\@seccntformat}[1]{\csname the#1\endcsname.}
\makeatother
\usepackage{amsthm}
\usepackage{bm}
\usepackage[colorlinks=true]{hyperref}
\begin{document}
%%%%%%%%%%%%% begin theorem definition %%%%%%%%%%%%%%%%%%
\newtheoremstyle{mytheorem}
  {\topsep}   % ABOVESPACE
  {\topsep}   % BELOWSPACE
  {\itshape}  % BODYFONT
  {}       % INDENT (empty value is the same as 0pt)
  {\bfseries} % HEADFONT
  {. }         % HEADPUNCT
  {5pt plus 1pt minus 1pt} % HEADSPACE
  { }          % CUSTOM-HEAD-SPEC
\newtheoremstyle{myremark}
  {\topsep}   % ABOVESPACE
  {\topsep}   % BELOWSPACE
  {\upshape}  % BODYFONT
  {}       % INDENT (empty value is the same as 0pt)
  {\bfseries} % HEADFONT
  {. }         % HEADPUNCT
  {5pt plus 0pt minus 1pt} % HEADSPACE
  {}          % CUSTOM-HEAD-SPEC\cite{}
\theoremstyle{mytheorem}
\newtheorem*{A}{Heyde theorem}
\newtheorem{theorem}{Theorem}[section]
 \newtheorem{theorema}{Theorem}
 \renewcommand{\thetheorema}{\Alph{theorema}}
 \newtheorem{proposition}[theorem]{Proposition}
 \newtheorem{lemma}[theorem]{Lemma}
\newtheorem{corollary}[theorem]{Corollary}
\newtheorem{definition}[theorem]{Definition}
\theoremstyle{myremark}
\newtheorem{remark}[theorem]{Remark}
\newtheorem{problem}[theorem]{Problem}
%%%%%%%%%%%%%%%%%%%%% end theorem definition %%%%%%%%%%%%%%%%%%
\noindent This work was accepted for publication \\
in the journal "Results in Mathematics"

\vskip 1 cm

\noindent{\Large\textbf{An Analogue of   Heyde's  Theorem for a Certain Class of Compact}}

\medskip

\noindent{\Large\textbf{Totally Disconnected Abelian Groups and $p$-quasicyclic Groups}}

\bigskip

\noindent{Gennadiy Feldman}  

\bigskip

\noindent B. Verkin Institute for Low Temperature Physics and Engineering\\
of the National Academy of Sciences of Ukraine, Kharkiv, Ukraine

\medskip

\noindent e-mail:    feldman@ilt.kharkov.ua

\medskip

\noindent ORCID ID  https://orcid.org/0000-0001-5163-4079 

\bigskip

\noindent\textbf{Abstract.} {According to the well-known   Heyde theorem, the  
Gaussian distribution  on the real line is characterized by the symmetry of 
the conditional  distribution of one linear form of independent random 
variables given another. In the article, we study an analogue  of this 
theorem for two
independent random variables  taking values either in a 
compact totally disconnected Abelian group of a certain class, which includes
finite cyclic groups and groups of $p$-adic integers, or in a $p$-quasicyclic group. 
In contrast to previous works devoted to group analogues of Heyde's theorem, 
we do not impose any restrictions on either coefficients of linear forms 
(they can be arbitrary topological automorphisms of the group) or the characteristic 
functions of random variables.
For the proof  we use methods of abstract harmonic analysis.}

\bigskip
\noindent \textbf{Mathematics Subject Classification.}   43A25,  43A35, 60B15, 62E10 

\bigskip

\noindent\textbf{Keywords.} Heyde  theorem;   
  automorphism; compact totally disconnected Abelian group; $p$-quasicyclic group;  
  group of $p$-adic integers

\section { Introduction}

According to the well-known   Heyde theorem,  the   Gaussian distribution  
on the real line is characterized by the symmetry of the conditional  
distribution of one linear form of independent random variables given another
\footnote{Let $\xi$ and $\eta$ be random variables.    
The conditional  distribution of  $\eta$ given $\xi$ is symmetric
if and only if the random vectors $(\xi, \eta)$ and $(\xi, -\eta)$ are 
identically distributed.} 
 (\!\!\cite{He}, see also \cite[Theorem 13.4.1]{KaLiRa}).   
Many studies have been devoted to analogues of Heyde's theorem for different 
classes of locally compact Abelian groups (see, e.g., \cite{Fe2,  Fe4, 
Fe20bb, My2,   FeTVP1, FeTVP, {M2013}, {M2020}, F_solenoid,   {Rima}, {M2023}, {Rima23}} 
and also
\cite[Chapter IV]{Febooknew}, where one can find additional references). 
Furthermore, special attention was paid to the case of two independent 
random variables.

Let us note the following. Suppose that $\xi_1$ and $\xi_2$ are
independent random variables with values in  a locally compact Abelian 
group $X$ and distributions $\mu_1$ and $\mu_2$.  Let $\alpha_j, \beta_j$ 
be topological automorphisms of $X$.   Assume that  the conditional  
distribution of the linear form $L_2 = \beta_1\xi_1 + \beta_2\xi_2$ 
given $L_1 = \alpha_1\xi_1 + \alpha_2\xi_2$ is symmetric. If we are 
interested in describing of 
the distributions $\mu_j$, then we can suppose, 
without loss of generality, that $L_1 = \xi_1 + \xi_2$ and 
$L_2 = \xi_1 + \alpha\xi_2$, where $\alpha$ is a topological 
automorphism of $X$. Taking into account this remark, for two independent 
random variables Heyde's theorem  can be 
 formulated as follows. 
\begin{A}Let $\xi_1$ and $\xi_2$ be
independent real-valued random variables with distributions
$\mu_1$ and $\mu_2$. Let $\alpha$ be a nonzero real number. 
Assume that the conditional  distribution of 
the linear form $L_2 = \xi_1 + \alpha\xi_2$ given $L_1 = \xi_1 + \xi_2$ 
is symmetric. Then the following statements hold:
\renewcommand{\labelenumi}{\rm(\roman{enumi})}
\begin{enumerate}
\item
if $\alpha\ne-1$,  then $\mu_1$ and $\mu_2$ are Gaussian distributions;
\item	
if $\alpha=-1$,  then $\mu_1=\mu_2=\mu$, where $\mu$ is an arbitrary distribution.
\end{enumerate} 
\end{A}

The simplest class of locally compact Abelian groups, 
where one can study characterization problems,
is the class of finite Abelian groups. 
Two different analogues of Heyde's theorem are proved
for finite Abelian groups  $X$ containing no elements of order 2. 
Let $\xi_1$ and $\xi_2$ be
independent random variables with values in  $X$ and distributions $\mu_1$ and $\mu_2$. 
Let
$\alpha$ be an automorphism of $X$. Consider two liner forms $L_1 = \xi_1 + \xi_2$ and 
$L_2 = \xi_1 + \alpha\xi_2$. Assume that  the conditional  distribution of  $L_2$ 
given $L_1$ is symmetric.
The first theorem states that if
$\mathrm{Ker}(I+\alpha)=\{0\}$, then $\mu_j$ are shifts of the Haar 
distribution on a subgroup of  $X$ (\!\!\cite[Theorem 1]{Fe2}, see 
also \cite[Theorem 10.2]{Febooknew}). According to the second theorem, 
if we do not impose any restrictions on the automorphism 
$\alpha$, but require that the characteristic functions of the 
distributions  
$\mu_j$ do not vanish, then $\mu_j$ are shifts of a distribution
supported in $\mathrm{Ker}(I+\alpha)$  (\!\!\cite[Theorem 2.1]{Rima}, 
see also \cite[Theorem 9.11]{Febooknew}). 

The main goal of this article is to prove an analogue of Heyde's theorem for compact
totally disconnected Abelian groups of a certain class, where
 in contrast to previous works, we do not impose any restrictions 
on $\alpha$ or the characteristic functions of $\mu_j$  (Theorem \ref{nth1}).
The theorems mentioned above for cyclic groups of odd order are 
consequences of Theorem \ref{nth1}. 
Moreover, an analogue of Heyde's theorem for groups of $p$-adic integers, where $p\ne 2$,
also follows from Theorem \ref{nth1}.  
It is important to note that 
a new class of distributions which has not previously 
been encountered in characterization problems on groups, appears in Theorem \ref{nth1}.
Then, based on Theorem \ref{nth1}, we prove an 
analogue of Heyde's 
theorem for discrete $p$-quasicyclic groups, where $p\ne 2$ (Theorem \ref{nth2}), 
also without any restrictions 
on $\alpha$ or the characteristic functions of $\mu_j$.   

In the article we use standard results of abstract harmonic analysis 
 (see e.g. \cite{Hewitt-Ross}). Let $X$ be a locally compact Abelian group. Denote by $Y$ the character
group of the group $X$. For  $x \in X$, denote by  $(x,y)$ the value of 
a character $y \in Y$ at the element  $x$. For a  subgroup $K$ of the group $X$,  denote by
 $A(Y, K) = \{y \in Y: (x, y) =1$  for all  $x\in
K\}$
its annihilator.

 Denote by $\mathrm{Aut}(X)$ the group
of all topological automorphisms of $X$  and by  $I$ the identity automorphism of a group.   
Let $G$ be a closed subgroup of $X$ and let $\alpha\in\mathrm{Aut}(X)$.
If $\alpha(G)=G$, i.e., the restriction  of  $\alpha$ to $G$ is a topological automorphism of the group   $G$, then we denote by   $\alpha_{G}$ this restriction.
A closed subgroup $G$ of $X$ is called characteristic if $\alpha(G)=G$ for 
all $\alpha\in\mathrm{Aut}(X)$. 
Let
 $\alpha:X\rightarrow X$ be a continuous endomorphism of the group $X$. 
 The adjoint endomorphism $\widetilde\alpha: Y\rightarrow Y$
is defined by the formula $(\alpha x,
y)=(x, \widetilde\alpha y)$ for all $x\in X$, $y\in
Y$.  Note that $\alpha\in\mathrm{Aut}(X)$ if and only 
if $\widetilde\alpha\in\mathrm{Aut}(Y)$. Let $p$ be a prime number. 
A group  $X$ is called  $p$-{group}  
if the order of every 
element of $X$ is a power of $p$. 
For a natural $n$, 
denote by $f_n:X \rightarrow X$ 
an  endomorphism of the group $X$
 defined by the formula  $f_nx=nx$, $x\in X$. Put $X^{(n)}=nX=f_n(X)$.   
Let $x\in X$ be an elements of finite 
order. Denote by $\langle x\rangle$ the subgroup of $X$ generated by $x$. 
Denote by $\mathbb{R}$ the group of real numbers and by
 $\mathbb{Z}(n)$ the  group of the integers 
modulo $n$, i.e., the finite cyclic group of order $n$.

Let $\mu$ and $\nu$ be probability 
distributions on the group   $X$. The convolution
$\mu*\nu$ is defined by the formula
$$
\mu*\nu(B)=\int\limits_{X}\mu(B-x)d \nu(x)
$$
for any Borel subset $B$ of $X$. 

Denote by
$$
\hat\mu(y) =
\int\limits_{X}(x, y)d \mu(x), \quad y\in Y,$$   the characteristic function 
(Fourier transform) of 
the distribution $\mu$.
 
 Define the distribution 
 $\bar \mu $ by the formula
 $\bar \mu(B) = \mu(-B)$ for any Borel  subset $B$ of $X$.
Then $\hat{\bar{\mu}}(y)=\overline{\hat\mu(y)}$. A distribution   
$\mu_1$ on the group $X$ is called a {factor} of $\mu$  if there is a distribution   
$\mu_2$ on $X$ such that $\mu=\mu_1*\mu_2$.
 
 We say that a function $a(y)$ on the group
 $Y$ is a characteristic function if there is a distribution $\mu$ 
 on the group $X$ such that
 $a(y)=\hat\mu(y)$ for all $y\in Y$. 
 Let $K$ be 
a compact subgroup of $X$. Denote by $m_K$ the Haar distribution 
on $K$. 
The
characteristic function  $\widehat m_K(y)$ is of the form
\begin{equation}
\label{fe22.1}\widehat m_K(y)=
\begin{cases}
1 & \text{\ if\ }\ \ y\in A(Y,K),
\\ 0 & \text{\ if\ }\ \ y\not\in
A(Y,K).
\end{cases}
\end{equation}

\section{ Main theorem}

Let $p$ be a prime number.
Recall the definitions of the group of $p$-adic integers $\mathbb{Z}_p$
and the $p$-quasicyclic group $\mathbb{Z}(p^\infty)$.
Consider a set of sequences of the form 
$x=(x_0, x_1, \dots, x_n, \dots)$, where $x_n\in\{0, 1, \dots, p-1\}$, 
and denote this set by $\mathbb{Z}_p$. The set $\mathbb{Z}_p$ is 
endowed with the 
product topology.
Each element
$x=(x_0, x_1, \dots, x_n, \dots)\in \mathbb{Z}_p$ is thought of as a
formal power series $\mathop{\sum}\limits_{n=0}^\infty x_np^n$.
The addition of formal power series defining in the usual way
corresponds to  the addition of the corresponding sequences. One
can define the multiplication in $\mathbb{Z}_p$ in the natural way as
the multiplication of formal power series. Then   $\mathbb{Z}_p$ is
transformed into a commutative ring and is called the ring of $p$-adic integers.
The element $e=(1, 0, \dots, 0,\dots)$ 
is the unit of the ring $\mathbb{Z}_p$. 
The additive group of the ring
 $\mathbb{Z}_p$ is called the group of $p$-adic integers and is also denoted as 
 $\mathbb{Z}_p$. Any nonzero closed subgroup $S$ of the group $\mathbb{Z}_p$ 
 is of the 
 form $S=p^k\mathbb{Z}_p$ for some nonnegative integer $k$.
 The family of the subgroups
 $\{p^k\mathbb{Z}_p\}_{k=0}^{\infty}$ forms an open 
 basis at the zero of
 the group $\mathbb{Z}_p$. The group $\mathbb{Z}_p$ is compact and totally 
 disconnected.
Denote by $\mathbb{Z}_p^\times$ 
the multiplicative group of all
invertible elements of the ring $\mathbb{Z}_p$. Then  
$\mathbb{Z}_p^\times=\{ c   =(c_0, c_1,\dots, c_n,   \dots)\in \mathbb{Z}_p:
c_0\ne 0\}$. Each element  $x \in
\mathbb{Z}_p$ is represented in the form $x = p^k c$, 
where $k$ is a nonnegative integer and
 $c \in \mathbb{Z}_p^\times$.  

Consider a set of rational numbers of the form $\{{k / p^n} : k=0, 1, \dots,p^n-1,\
	n=0,1,\dots\}$ and denote  by ${\mathbb Z}(p^\infty)$ this set. If
we define the operation in ${\mathbb Z}(p^\infty)$ as addition
modulo 1, then ${\mathbb Z}(p^\infty)$ is transformed into an
Abelian group. The group ${\mathbb Z}(p^\infty)$ is endowed with  
the discrete topology.
This group is called  the $p$-quasicyclic  group.
For a fixed $n$ consider a subgroup
of ${\mathbb Z}(p^\infty)$ consisting of all elements of the form
$\{{k / p^n} : k=0, 1, \dots,p^n-1\}$. This subgroup is isomorphic to 
the group ${\mathbb Z}(p^n)$.  To avoid introducing new notation, 
we will denote this subgroup by ${\mathbb Z}(p^n)$. Any 
proper subgroup of the group
${\mathbb Z}(p^\infty)$ coincides with ${\mathbb Z}(p^n)$ 
for some nonnegative integer $n$.
 The group  ${\mathbb Z}(p^\infty)$ is isomorphic to the multiplicative group
of  $p^n$th roots of unity, where $n$ goes through the nonnegative
integers, which is endowed with  the discrete topology. 

The character group of the group $\mathbb{Z}_p$ is topologically 
 isomorphic
 to the group  ${\mathbb Z}(p^\infty)$  and the value of a character 
 ${y=k / p^n\in	{\mathbb Z}(p^\infty)}$
  at an element $x=(x_0, x_1,\dots, x_n,\dots )\in \mathbb{Z}_p$
   is given by the formula
$$(x, y)=\exp\displaystyle{\left\{{\Bigl(x_0+x_1 p + \dots
	+x_{n-1}p^{n-1}\Bigr)\frac{2\pi ik}{p^n}}\right\}}.$$ 

The  groups $\mathrm{Aut}({\mathbb{Z}_p})$ and $\mathrm{Aut}({\mathbb Z}(p^\infty))$ 
 are isomorphic to $\mathbb{Z}_p^\times$. Let $\alpha\in\mathrm{Aut}({\mathbb{Z}_p})$. 
Then $\alpha$ corresponds to an element 
$c=(c_0, c_1, \dots, c_n,\dots) \in\mathbb{Z}_p^\times$
such that $\alpha$ acts on the group ${\mathbb{Z}_p}$ 
as the multiplication by $c$, i.e., 
  $\alpha x=cx$ for all $x\in \mathbb{Z}_p$.
The adjoint  automorphism
 $\widetilde\alpha\in \mathrm{Aut}({\mathbb Z}(p^\infty))$
  acts on ${\mathbb Z}(p^\infty)$     
as follows. Put
$s_n=c_0+c_1p+
c_2p^2+\dots+c_{n-1}p^{n-1}$. The restriction of 
 $\widetilde\alpha$ to the subgroup ${\mathbb{Z}}(p^n) \subset
{\mathbb{Z}}(p^\infty)$ is of the form $\widetilde\alpha y = s_ny$ for all $y \in
\mathbb{Z}(p^n)$, i.e., $\alpha$ acts on ${\mathbb{Z}}(p^n)$ as
the multiplication by $s_n$.  

  If $\alpha$ corresponds
to an element $c=(c_0, c_1, \dots, c_n,\dots) \in\mathbb{Z}_p^\times$, we will write
$\alpha=(c_0, c_1, \dots, c_n,\dots)$.
  We note that any closed subgroup of the group $\mathbb{Z}_p$
 is characteristic. The same is true for any  
 subgroup of the group  ${\mathbb Z}(p^\infty)$.
 
 Let $\{G_\iota: \iota\in {\mathcal I}\}$ be a nonvoid family of
compact Abelian groups. Denote by $\mathop\mathbf{P}\limits_{\iota \in
	{\mathcal I}}G_\iota$  the {direct product of the groups}
$G_\iota$ considering in the product topology. 
Let $\{H_\iota: \iota\in {\mathcal I}\}$ be a nonvoid family of
discrete Abelian groups. Denote by
$\mathop{\mathbf{P}^*}\limits_{\iota \in {\mathcal
		I}}H_\iota$ the {weak direct product of the
	groups} $H_\iota$, considering in the discrete topology.
	
	We begin by proving an analogue of Heyde's theorem for a certain class of compact
totally disconnected Abelian groups. This class includes, in particular,  
finite cyclic groups of odd order and groups of $p$-adic integers, where $p\ne 2$.

\begin{theorem}\label{nth1}  Let $X$ be a compact
totally disconnected Abelian group of the form
\begin{equation}\label{e20.35}
X=\mathop\mathbf{P}\limits_{p_j\in {\mathcal P}}{X}_{p_j},
\end{equation}
where ${\mathcal P}$ is a set of pairwise distinct prime numbers such that 
$2\notin {\mathcal P}$ and $X_{p_j}$ is either the cyclic $p_j$-group 
$\mathbb{Z}(p_j^{k_j})$ or the group of $p_j$-adic integers 
$\mathbb{Z}_{p_j}$.
Let  $\alpha$ be a topological automorphism of the group $X$. 
 Let $\xi_1$ and $\xi_2$ be
independent random variables with values in   $X$ and distributions
$\mu_1$ and $\mu_2$. Assume that the conditional  distribution of 
the linear form $L_2 = \xi_1 + \alpha\xi_2$ given $L_1 = \xi_1 + \xi_2$ 
is symmetric.
Then  there is a compact subgroup $G$ of the group $X$ satisfying the condition 
$(I-\alpha)(G)=G$ 
and a distribution  
$\lambda$ supported in $G$ such that 
the following statements are true:
\renewcommand{\labelenumi}{\rm(\roman{enumi})}
\begin{enumerate}
\item
$\mu_j$ are shifts of $\lambda$;
	\item
$G$ is the minimal subgroup containing the support of 
$\lambda$;
\item
the Haar distribution $m_{(I+\alpha)(G)}$ is a factor of  $\lambda$;
\item
if   $\eta_j$
are independent identically distributed random variables with values in
$G$  and distribution  $\lambda$, then the conditional 
distribution of the linear form $M_2=\eta_1 + \alpha_G\eta_2$ 
given $M_1=\eta_1 + \eta_2$  is symmetric.
\end{enumerate} 
\end{theorem}

To prove Theorem \ref{nth1} we need the following lemmas.
\begin{lemma}[{\!\!\protect\cite[Lemma 9.1]{Febooknew}}]
\label{lem1}
 Let $X$ be a second countable locally compact Abelian group with character group
 $Y$ and let $\alpha$ be a topological automorphism of $X$.
Let
$\xi_1$ and  $\xi_2$  be independent random variables with values in
 the group $X$  and distributions $\mu_1$ and $\mu_2$.   The conditional 
 distribution of the linear form $L_2 = \xi_1 + \alpha\xi_2$ given 
 $L_1 = \xi_1 + \xi_2$ is symmetric if and only
 if the characteristic functions
 $\hat\mu_j(y)$ satisfy the equation
\begin{equation}\label{11.04.1}
\hat\mu_1(u+v )\hat\mu_2(u+\widetilde\alpha v )=
\hat\mu_1(u-v )\hat\mu_2(u-\widetilde\alpha v), \quad u, v \in Y.
\end{equation}
\end{lemma}
Equation (\ref{11.04.1}) is called Heyde's functional equation. 
Thanking Lemma \ref{lem1} the proof of Theorem \ref{nth1} is reduced 
to the description of all solutions of   functional equation (\ref{11.04.1}) 
on the character group of the group $X$ of the form (\ref{e20.35})
 in the class of characteristic functions.

It is convenient for us to formulate   the following   well-known statement  
in the form as a lemma (for the proof see, e.g., \cite[Proposition 2.10]{Febooknew}).
\begin{lemma}\label{lem2}  Let $X$ be a locally compact Abelian group  
with character group
 $Y$
and let $\mu$ be a distribution on $X$. The sets 
$$E=\{y\in Y:  \hat\mu(y)=1\}, \quad
B=\{y\in Y:  |\hat\mu(y)|=1\}$$ are closed subgroups of the group $Y$ and 
the distribution   $\mu$ is supported in  $A(X,E)$.
\end{lemma}

\begin{lemma}[{\!\!\protect\cite[Lemma 2.5]{Rima}, see also 
		\cite[Lemma 9.10]{Febooknew}}]
\label{lem11} 
Let $X$ be a second countable locally compact Abelian group with character group
 $Y$. Let $S$ be a closed subgroup of  $Y$ such that  $S^{(2)}=S$.   Put   $G=A(X, S)$.
 Let  $\alpha$ be a topological automorphism of the group   $X$ such that 
 $\alpha(G)=G$.
Let
$\xi_1$ and  $\xi_2$  be independent random variables with values in
 $X$  and distributions $\mu_1$ and $\mu_2$ such that
$$|\hat\mu_1(y)|=|\hat\mu_2(y)|=1, \quad y\in S.$$  
Assume that the conditional  distribution of the linear form $L_2 = \xi_1 + \alpha\xi_2$ given $L_1 = \xi_1 + \xi_2$ is symmetric. Then there are some shifts $\lambda_j$ of the distributions $\mu_j$  such that $\lambda_j$ are supported in $G$. In doing so, if  $\eta_j$ are independent random variables with values in   $G$  and distributions $\lambda_j$, then the conditional  distribution of the linear form
$M_2=\eta_1 + \alpha_G\eta_2$ given $M_1=\eta_1 + \eta_2$ is symmetric.
\end{lemma}

The proof of the following lemma essentially repeats the proof of item
2 of \cite[Theorem 2.1]{Rima}, where the lemma was proved under the 
assumption that $Y$ is a finite Abelian group.
For an arbitrary Abelian group $Y$, we use the notation $\mathrm{Aut}(Y)$ 
for the group of all automorphisms of $Y$.

\begin{lemma}\label{newle1} 
Let $Y$ be an Abelian group and let $\beta$ be 
an automorphism of $Y$ such that $I-\beta\in\mathrm{Aut}(Y)$.  
Let $f(y)$ and $g(y)$ be functions on the group $Y$ satisfying the equation
\begin{equation} \label{15.08.1}
f(u+v)g(u+\beta v)=f(u-v)g(u-\beta v), \quad u, v\in Y.
\end{equation}
Then $f(y)$ and $g(y)$ satisfy the equations
\begin{equation}\label{11.04.16}
f(y)=f(-(I+\beta)(I-\beta)^{-1} y)
g(-2\beta(I-\beta)^{-1} y), 
\quad y\in Y, 
\end{equation}
\begin{equation}\label{11.04.8}
g(y)=g((I+\beta)(I-\beta)^{-1} y)f(2(I-\beta)^{-1} y), 
\quad y\in Y. 
\end{equation}
Assume that the inequalities
$0\le f(y)\le 1$, $0\le g(y)\le 1$, $y\in Y$, are valid.
Put $\kappa=-f_4\beta(I-\beta)^{-2}$. Let 
$\kappa^my_0=y_0$ for some $y_0\in Y$ and some natural $m$.
Then  
 \begin{equation}\label{11.04.14}
f(y_0)=g(-2\beta(I-\beta)^{-1} y_0), 
\end{equation} 
\begin{equation}\label{11.04.15}
g(y_0)=f(2(I-\beta)^{-1} y_0). 
\end{equation} 
\end{lemma}
\begin{proof}
Substituting  first $u=\beta y$, $v=y$   and then $u=v=y$  
in equation (\ref{15.08.1}), we obtain respectively
\begin{equation}\label{11.04.10}
f((I+\beta) y)g(2\beta y)=f(-(I-\beta) y), 
\quad y\in Y,
\end{equation}
\begin{equation}\label{11.04.7}
f(2y)g((I+\beta) y)=g((I-\beta) y), \quad y\in Y. 
\end{equation}
Substituting first 
 $-(I-\beta)^{-1} y$ instead of $y$  in equation (\ref{11.04.10})  
 and then $(I-\beta)^{-1} y$ instead of $y$  
 in equation (\ref{11.04.7}), we receive  (\ref{11.04.16}) and (\ref{11.04.8}).

 Inasmuch as $0\le f(y)\le 1$ and $0\le g(y)\le 1$ for all $y\in Y$, we find
from (\ref{11.04.16}) and (\ref{11.04.8}) respectively  
\begin{equation}\label{11.04.11}
f(y)\le g(-2\beta(I-\beta)^{-1} y), \quad y\in Y, 
\end{equation}
\begin{equation}\label{11.04.9}
g(y)\le f(2(I-\beta)^{-1} y), \quad y\in Y. 
\end{equation}
From (\ref{11.04.11}) and (\ref{11.04.9}) we obtain
\begin{equation}\label{11.04.18}
f(y)\le g(-2\beta(I-\beta)^{-1} y)\le f(\kappa y), 
\quad y\in Y.
\end{equation}
Hence
\begin{equation}\label{11.04.17}
f(y)\le f(\kappa y), \quad y\in Y.
\end{equation}

Assume that   
  $\kappa^my_0=y_0$ for some $y_0\in Y$ and some natural $m$.  
  We find from  (\ref{11.04.17})  the inequalities
$$
f(y_0)\le f(\kappa y_0)\le \dots \le f(\kappa^{m-1} y_0)
\le f(\kappa^{m} y_0).
$$
Since $f(\kappa^{m} y_0)=f(y_0)$, we get
\begin{equation}\label{11.04.13}
f(y_0)=f(\kappa y_0)  
\end{equation}
and  equality (\ref{11.04.14}) 
follows from (\ref{11.04.18}) and (\ref{11.04.13}). 

An analogous statement for the function   $g(y)$ also holds. Indeed, 
it follows from (\ref{11.04.11}) and (\ref{11.04.9})  that

\begin{equation}\label{11.04.12}
g(y)\le f(2(I-\beta)^{-1} y)\le g(\kappa y), \quad y\in Y.
\end{equation} 
Hence
\begin{equation*}\label{n11.04.13}
g(y)\le g(\kappa y), \quad y\in Y.
\end{equation*} 
We get from here the inequalities
$$
g(y_0)\le g(\kappa y_0)\le \dots \le g(\kappa^{m-1} y_0)
\le g(\kappa^{m} y_0).
$$
This implies that
\begin{equation}\label{n11.04.18}
g(y_0)=g(\kappa y_0) 
\end{equation}
and equality (\ref{11.04.15}) 
follows from (\ref{11.04.12}) and (\ref{n11.04.18}). 
\end{proof}

\noindent\textit{Proof of Theorem} \ref{nth1}. 
Before proceeding to the proof of the theorem, let us make the following remarks.
Denote by $Y$ the character group of the group $X$ and by 
$Y_{p_j}$  the character group of the group $X_{p_j}$. The group 
$Y_{p_j}$ is topologically isomorphic to either the group 
 $\mathbb{Z}(p_j^{k_j})$ or
 $\mathbb{Z}(p_j^\infty)$,  and the group
 $Y$  is topologically isomorphic to the weak direct product of the group
$Y_{p_j}$. To avoid introducing additional notation, we assume that
\begin{equation}\label{1e20.35}
Y=\mathop\mathbf{P^*}\limits_{p_j\in {\mathcal P}}{Y}_{p_j},
\end{equation}  
where $Y_{p_j}$ is  either the group  $\mathbb{Z}(p_j^{k_j})$ or
 $\mathbb{Z}(p_j^\infty)$. A standard reasoning shows that any closed subgroup
of a group of the form (\ref{e20.35}) is topologically isomorphic to
a group of the form 
\begin{equation*} 
K=\mathop{ \mathbf{P}}\limits_{p_j \in {\mathcal S}}K_{p_j},
\end{equation*}
where $\mathcal S\subset \mathcal P$  and $K_{p_j}$ is a closed 
subgroup of $X_{p_j}$, i.e., $K$ is also a group of the form (\ref{e20.35}).
Moreover, $K$ is a characteristic subgroup of $X$. Similar statements are 
also true for the group
$Y$.
 
We divide the proof of the theorem into several steps. 
In items 1 and 2 we follow the scheme of the proof of  \cite[Theorem 2.1]{Rima}.
 
1. We show in this item that  there is a compact subgroup $G$ of 
the group $X$ 
such that $(I-\alpha)(G)=G$ and $\mu_j$ are 
shifts of some distributions $\lambda_j$
supported in  $G$. In so doing, 
 if  $\eta_j$ are independent random 
variables with values in $G$  and distributions $\lambda_j$, then 
the conditional  distribution of the linear form
$M_2=\eta_1 + \alpha_G\eta_2$ given $M_1=\eta_1 + \eta_2$ is symmetric.

Put
$$
S=\{y\in Y:|\hat\mu_1(y)|=|\hat\mu_2(y)|=1\}.
$$
By Lemma \ref{lem2}, $S$ is a  subgroup of the group  $Y$. 
It follows from  (\ref{1e20.35}) that $f_2\in\mathrm{Aut}(Y)$. 
This implies that 
 $S^{(2)}=S$. Set   $G=A(X, S)$. We have  $\alpha(G)=G$  because any
closed subgroup  of the group $X$ is characteristic.      
 All conditions of Lemma \ref{lem11} are fulfilled. 
By Lemma \ref{lem11},  we can  replace  the distributions  $\mu_j$ by their shifts  
$\lambda_j$   in such a way that   $\lambda_j$ are supported in $G$,
and   if   $\eta_j$
are independent random variables with values in
 $G$  and distributions $\lambda_j$, then the conditional 
distribution of the linear form $M_2=\eta_1 + \alpha_G\eta_2$ given 
$M_1=\eta_1 + \eta_2$  is symmetric.

Let us verify that $(I-\alpha)(G)=G$.
 Denote by $H$ the character group of the group $G$. We will 
 consider $\lambda_j$ as distributions on the group $G$. Then
\begin{equation}\label{y11.04.2}
\{h\in H:|\hat\lambda_1(h)|=|\hat\lambda_2(h)|=1\}=\{0\}.
\end{equation}
By Lemma \ref{lem1}, 
the characteristic functions $\hat\lambda_j(h)$   satisfy equation 
(\ref{11.04.1}) on the group $H$.
Put $$T=\mathrm{Ker}(I-\widetilde\alpha_G).$$    
It follows from $\widetilde\alpha_G h=h$ for all $h\in T$ that  
the restriction of equation (\ref{11.04.1}) for the characteristic functions
$\hat\lambda_j(h)$ to the subgroup $T$ takes the form
\begin{equation}\label{11.04.2}
\hat\lambda_1(u+v)\hat\lambda_2(u+v)=\hat\lambda_1(u-v)\hat\lambda_2(u-v), \quad u, v\in T.
\end{equation}
Substituting  $u=v=h$ in  equation (\ref{11.04.2}), we get 
\begin{equation}\label{n11.04.2}
\hat\lambda_1(2h)\hat\lambda_2(2h)=1, \quad h\in T.
\end{equation}
Since the subgroup $G$ is topologically isomorphic to
a group of the form (\ref{e20.35}), we have $f_2\in\mathrm{Aut}(H)$. 
 It follows from this that
 ${T^{(2)}}=T$ and (\ref{n11.04.2}) implies that 
\begin{equation}\label{ynn11.04.2}
|\hat\lambda_1(h)|=|\hat\lambda_2(h)|=1, \quad h\in T.
\end{equation}
We conclude from (\ref{y11.04.2}) and (\ref{ynn11.04.2})  
 that $T=\{0\}$. 
On the one hand, this implies that
$A(G, \mathrm{Ker}(I-\widetilde\alpha_G))=
A(G, T)=G$. On the other hand, 
$A(G, \mathrm{Ker}(I-\widetilde\alpha_G))=(I-\alpha_G)(G)=(I-\alpha)(G)$.
We proved that $(I-\alpha)(G)=G$.

2.   Put   $\nu_j=\lambda_j*\bar\lambda_j$. Then 
$\hat\nu_j(h)=|\hat\lambda_j(h)|^2\ge 0$ for all   $h\in H$, $j=1, 2$. 
Set $$f(h)=\hat\nu_1(h), \quad g(h)=\hat\nu_2(h), \quad h\in H.$$ 

As has been proven in item 1, $T=\mathrm{Ker}(I-\widetilde\alpha_G)=\{0\}$.
Since the subgroup $H$ is topologically isomorphic to a group of the form 
(\ref{1e20.35}), this implies that   $f_2, I-\widetilde\alpha_G\in \mathrm{Aut}(H)$. 
It follows from this that   
$\kappa=-f_4\widetilde\alpha_G(I-\widetilde\alpha_G)^{-2}\in \mathrm{Aut}(H)$. 
Take an element $h\in H$. Inasmuch as $H$ is topologically isomorphic to a group of the form 
(\ref{1e20.35}), there is $m$ depending, generally speaking, on $h$ such that
 $\kappa^mh=h$.
By Lemma \ref{lem1}, the characteristic functions $f(h)$ and $g(h)$ satisfy
the equation 
\begin{equation}\label{new1}
f(u+v)g(u+\widetilde\alpha_G v)=f(u-v)g(u-\widetilde\alpha_G v), \quad u, v\in H.
\end{equation}
Thus, all conditions of Lemma \ref{newle1},   where $Y=H$, 
 $\beta=\widetilde\alpha_G$, and $y_0$ is an arbitrary element of
 $H$, are fulfilled. By Lemma \ref{newle1}, the 
functions $f(h)$ and $g(h)$ satisfy  equations
(\ref{11.04.16}) and (\ref{11.04.8}) and equalities 
(\ref{11.04.14}) and (\ref{11.04.15})
are true for all $h\in H$.
 
3. Take $h\in H$. We prove in this item that
if either $f(h)\ne 0$ or $g(h)\ne 0$, then  
$h\in \mathrm{Ker}(I+\widetilde\alpha_G)$.

3a. Suppose that $f(h_1)\ne 0$. 
Consider the subgroup $\langle h_1\rangle$ of the group $H$ generated by
the element $h_1$. Put 
$$
h_2=-2\widetilde\alpha_G(I-\widetilde\alpha_G)^{-1} h_1.
$$  
We have
$-f_2\widetilde\alpha_G(I-\widetilde\alpha_G)^{-1}\in\mathrm{Aut}(H)$. 
Since $H$ is topologically isomorphic to a group of the form 
(\ref{1e20.35}), each subgroup of $H$ is characteristic.  
This implies that $h_2\in \langle h_1\rangle$. Moreover, the elements $h_1$ 
and $h_2$ have the same order  and hence $\langle h_1\rangle=
\langle h_2\rangle$. 
Since $f(h_1)\ne 0$, it follows from (\ref{11.04.14}) that $g(h_2)\ne 0$. 
Set 
\begin{equation}\label{nn11.04.15}
l_1=-(I+\widetilde\alpha_G)(I-\widetilde\alpha_G)^{-1} h_1.
\end{equation}
We find from   (\ref{11.04.16}) and (\ref{11.04.14}) that  then $f(l_1)=1$.  
By Lemma \ref{lem2}, the set $\{h\in H: f(h)=1\}$ is a subgroup of $H$. 
As far as $f(l_1)=1$, we conclude that
$f(h)=1$ at each element $h$ of the subgroup $\langle l_1\rangle$.
Put
\begin{equation}\label{nnn11.04.15}
l_2=(I+\widetilde\alpha_G)(I-\widetilde\alpha_G)^{-1} h_2.
\end{equation}
Since $g(h_2)\ne 0$, we find from  (\ref{11.04.8}) and (\ref{11.04.15})   
that  
$g(l_2)=1$. Taking into account that the set $\{h\in H: g(h)=1\}$ 
is a subgroup  of $H$, 
we see that
$g(h)=1$ at each element of the subgroup $\langle l_2\rangle$.
As far as the elements $h_1$ and $h_2$ have the same order, 
it follows from (\ref{nn11.04.15}) and  (\ref{nnn11.04.15}) that the elements 
$l_1$ and $l_2$ also have the same order. 
Hence
$\langle l_1\rangle=\langle l_2\rangle$. 
Thus, we proved that $f(h)=g(h)=1$ at each element $h$ of the subgroup
$\langle l_1\rangle$. 

It follows from (\ref{y11.04.2}) 
that if
$f(h)=g(h)=1$ at an element $h\in H$, then $h=0$. If we take 
into consideration that
$f(h)=g(h)=1$ at each element $h$ of the subgroup
$\langle l_1\rangle$, we conclude that $\langle l_1\rangle=\{0\}$, i.e., 
$l_1=0$. Taking into account (\ref{nn11.04.15}) and the fact that 
$(I-\widetilde\alpha_G)^{-1}\in\mathrm{Aut}(H)$, we see that  
 if $f(h_1)\ne 0$ at  an element $h_1\in H$, then 
 $h_1\in \mathrm{Ker}(I+\widetilde\alpha_G)$.

3b. Suppose now that $g(m_1)\ne 0$ and 
argue in the same way as in subitem  3a. 
Consider the subgroup $\langle m_1\rangle$ of $H$ generated by
the element $m_1$. Put $$m_2=2(I-\widetilde\alpha_G)^{-1} m_1.$$  
As far as $f_2(I-\widetilde\alpha_G)^{-1}\in\mathrm{Aut}(H)$ 
and each subgroup of 
$H$ is characteristic, we conclude that
$m_2\in \langle m_1\rangle$. Moreover, the elements 
$m_1$ and $m_2$ have the same order  and hence 
$\langle m_1\rangle=\langle m_2\rangle$. It follows from
(\ref{11.04.15}) that $f(m_2)\ne 0$. 
Put 
\begin{equation}\label{x1}
n_1=(I+\widetilde\alpha_G)(I-\widetilde\alpha_G)^{-1} m_1.
\end{equation}
We find from   (\ref{11.04.8}) and (\ref{11.04.15}) that  then 
$g(n_1)=1$.  
Inasmuch as the set $\{m\in H: g(m)=1\}$ is a subgroup  of $H$,
$g(m)=1$ at each element $m$ of the subgroup $\langle n_1\rangle$.
Put 
\begin{equation}\label{x2}
n_2=-(I+\widetilde\alpha_G)(I-\widetilde\alpha_G)^{-1} m_2.
\end{equation}
In view of $f(m_2)\ne 0$, we find from  (\ref{11.04.16}) and  (\ref{11.04.14})   
that  
$f(n_2)=1$.   Arguing for the function $f(h)$ in the same way as in subitem 
3a for the function 
$g(h)$, we come to the conclusion that
$f(m)=1$ at each element $m$ of the subgroup $\langle n_2\rangle$.
Since the elements $m_1$ and $m_2$ have the same order, in view of (\ref{x1}) 
and  (\ref{x2}), the elements $n_1$ and $n_2$ also  
have the same order, and we have
$\langle n_1\rangle=\langle n_2\rangle$. 
Thus, we proved that $f(h)=g(h)=1$ at each element of the subgroup
$\langle n_1\rangle$. 

It follows from (\ref{y11.04.2}) 
that if
$f(h)=g(h)=1$ at an element $h\in H$, then $h=0$. Due to the fact that
$f(h)=g(h)=1$ at each element $h$ of the subgroup
$\langle n_1\rangle$, we conclude that $\langle n_1\rangle=\{0\}$, 
i.e., $n_1=0$.  Taking into account (\ref{x1}) and the fact that 
$(I-\widetilde\alpha_G)^{-1}\in\mathrm{Aut}(H)$, we see that  
 if $g(m_1)\ne 0$ at  an element $m_1\in H$, then 
 $m_1\in \mathrm{Ker}(I+\widetilde\alpha_G)$. 

4. We prove in this item that 
$\lambda_1=\lambda_2$. From what was proven in item 3 it follows
that  the characteristic 
functions $\hat\lambda_j(h)$ are represented in the form
\begin{equation}\label{nn11.04.16}
\hat\lambda_j(h) = \begin{cases}a_j(h)  &\text{\ if\ }\ \ 
h\in \mathrm{Ker}(I+\widetilde\alpha_G), \\ 0& 
 \text{\ if\ }\ \ h\notin \mathrm{Ker}(I+\widetilde\alpha_G),
\\
\end{cases}
\end{equation}
where $a_j(h)$ are some characteristic functions on 
the group $\mathrm{Ker}(I+\widetilde\alpha_G)$.
As as noted in the proof of item 1, the characteristic functions $\hat\lambda_j(h)$   satisfy equation (\ref{11.04.1}). Consider the restriction of equation (\ref{11.04.1}) for the 
characteristic functions $\hat\lambda_j(h)$ to the subgroup 
$\mathrm{Ker}(I+\widetilde\alpha_G)$. Since $\widetilde\alpha_G h=-h$ for all 
$h\in \mathrm{Ker}(I+\widetilde\alpha_G)$, we find from  (\ref{nn11.04.16}) 
that the characteristic functions $a_j(h)$ satisfy the equation
\begin{equation}\label{y}
a_1(u+v)a_2(u-v)=a_1(u-v)a_2(u+v), \quad u, v\in \mathrm{Ker}(I+\widetilde\alpha_G).
\end{equation}
Substituting $u=v=h$  in equation (\ref{y}) and taking into account 
that the subgroup $\mathrm{Ker}(I+\widetilde\alpha_G)$ contains 
no elements of order 2, we conclude
that $a_1(h)=a_2(h)$ for all  $h\in \mathrm{Ker}(I+\widetilde\alpha_G)$. 
Put $a(h)=a_1(h)=a_2(h)$.
In view of
(\ref{nn11.04.16}), this implies that $\hat\lambda_1(h) 
= \hat\lambda_2(h)$, $h\in H$.
Hence $\lambda_1=\lambda_2$. Put $\lambda=\lambda_1=\lambda_2$. The characteristic 
function $\hat\lambda(h)$
is of the form
\begin{equation}\label{z}
\hat\lambda(h) = \begin{cases}a(h)  
&\text{\ if\ }\ \ h\in \mathrm{Ker}(I+\widetilde\alpha_G), \\ 0& 
\text{\ if\ }\ \ h\notin \mathrm{Ker}(I+\widetilde\alpha_G).
\\
\end{cases}
\end{equation}
Taking into account what was proved in item 1, 
statements (i) and (iv) are proved. Moreover, 
 it follows from (\ref{y11.04.2}) 
that statement (ii) is also   proved.

5. To complete the proof of the theorem, we prove in this item 
that statement (iii) holds. 

We note that $A(H, (I+\alpha_G)(G))=\mathrm{Ker}(I+\widetilde\alpha_G)$. 
Taking into account
(\ref{fe22.1}), this implies that the characteristic function of the  
Haar distribution $m_{(I+\alpha_G)(G)}$ is of the form
\begin{equation}
\label{nfe22.1}\widehat m_{(I+\alpha_G)(G)}(h)=
\begin{cases}
1 & \text{\ if\ }\ \ h\in \mathrm{Ker}(I+\widetilde\alpha_G),
\\ 0 & \text{\ if\ }\ \ h\not\in
\mathrm{Ker}(I+\widetilde\alpha_G).
\end{cases}
\end{equation}
It follows from (\ref{z}) and (\ref{nfe22.1}) that 
$\hat\lambda(h)=\hat\lambda(h)\widehat m_{(I+\alpha_G)(G)}(h)$ for all $h\in H$.
Hence $\lambda=\lambda*m_{(I+\alpha_G)(G)}$, i.e.,
 the Haar distribution $m_{(I+\alpha_G)G}$ is a factor of the 
 distributions $\lambda$. In view of $(I+\alpha)(G)=(I+\alpha_G)(G)$, 
statement  (iii) is proved and hence the theorem is completely proved. 
   $\hfill\Box$
   
\medskip

The following statement results from the proof of Theorem \ref{nth1}.

\begin{corollary}\label{nnco1} Let $X$ be a compact
totally disconnected Abelian group of the form
$(\ref{e20.35})$ with character group
 $Y$. Assume that all conditions of 
Theorem $\ref{nth1}$ are fulfilled and the characteristic functions $\hat\mu_j(y)$ satisfy the condition
$$
\{y\in Y:|\hat\mu_1(y)|=|\hat\mu_2(y)|=1\}=\{0\}.
$$
Then $(I-\alpha)(X)=X$, $\mu_1=\mu_2=\mu$, $X$ is the minimal subgroup containing 
the support of 
$\mu$, and the Haar distribution $m_{(I+\alpha)(X)}$ 
is a factor of of $\mu$.
\end{corollary}

Let $X$ be a finite 
   cyclic group of odd order. Then $X$ is isomorphic to 
   a group of the form (\ref{e20.35}) and Theorem \ref{nth1}
   implies the following statement. 

\begin{corollary}\label{a1}  Let $X$ be a finite cyclic group of 
odd order and let  $\alpha$ be an automorphism of $X$.  
 Let $\xi_1$ and $\xi_2$ be independent random variables with values in 
 the group  $X$ 
and distributions $\mu_1$ and $\mu_2$. Assume that the conditional  
distribution of the linear form $L_2 = \xi_1 + \alpha\xi_2$ given 
$L_1 = \xi_1 + \xi_2$ is symmetric. 
Then there is a subgroup $G$ of the group $X$
satisfying the condition $(I-\alpha)(G)=G$  and a distribution  
$\lambda$ supported in $G$ such that 
statements {\rm (i)--(iv)} of Theorem {\rm\ref{nth1}} are true.
\end{corollary}

\begin{corollary}\label{co1} Let $X$ be a finite cyclic group of odd order and let
 $\alpha$ be an automorphism of $X$.  Let $\xi_1$ and $\xi_2$ be
independent random variables with values in the group  $X$ and distributions
$\mu_1$ and $\mu_2$. Assume that the conditional  distribution of 
the linear form $L_2 = \xi_1 + \alpha\xi_2$ 
given $L_1 = \xi_1 + \xi_2$ is symmetric. 
If $\mathrm{Ker}(I+\alpha)=\{0\}$, then there is a subgroup $G$ of the group $X$ such that 
  $\mu_j$ are shifts of the Haar distribution $m_G$.
\end{corollary}
\begin{proof} 
   By Corollary \ref{a1}, there is a subgroup $G$ of the group $X$ such 
   that  the distributions $\mu_j$ are shifts of 
a distribution $\lambda$ supported in $G$. 
 Inasmuch as $\mathrm{Ker}(I+\alpha)=\{0\}$ and $X$ 
is a finite group, 
we have $I+\alpha\in\mathrm{Aut}(X)$. This implies that $(I+\alpha)(G)=G$. 
Hence $m_{(I+\alpha)(G)}=m_{G}$. By Corollary \ref{a1}, the Haar distribution  
$m_{(I+\alpha)(G)}$ is a factor of
$\lambda$. It follows from this that the Haar distribution  $m_G$ is a factor of
$\lambda$. Taking into account that the distribution $\lambda$ is supported in $G$, 
we have $\lambda=m_G$. 
\end{proof}
We note that Corollary \ref{co1} is a particular case  of Theorem 1 in \cite{Fe2}, see also 
\cite[Theorem 10.2]{Febooknew} which refers to arbitrary finite Abelian groups of odd order. Theorem 1 was proved 
in \cite{Fe2} in a different way.

\begin{corollary}\label{15.08co1} Let $X$ be a compact
totally disconnected Abelian group of the form
$(\ref{e20.35})$ with character group
 $Y$. Assume that all conditions of 
Theorem $\ref{nth1}$ are fulfilled and the characteristic functions $\hat\mu_j(y)$
do not vanish. Then $\mu_j$ are shifts of a distribution supported in 
$\mathrm{Ker}(I+\alpha)$.
\end{corollary}

\begin{proof} By Theorem \ref{nth1}, there is a compact
 subgroup $G$ of the group $X$ such 
   that  the distributions $\mu_j$ are shifts of 
a distribution $\lambda$ supported in $G$. 
 Denote by $H$ the character group of the group $G$.
 By Theorem \ref{nth1}, the characteristic function 
 $\hat\lambda(h)$  is of the form 
 (\ref{z}), where $a(h)$ is  a characteristic function on the subgroup $\mathrm{Ker}(I+\widetilde\alpha_G)$.  Since the characteristic functions $\hat\mu_j(y)$ do not vanish,   
the characteristic function  $\hat\lambda(h)$ also does not vanish.
Taking into account (\ref{z}), this means that 
$\mathrm{Ker}(I+\widetilde\alpha_G)=H$, 
 i.e., $\widetilde\alpha_G h=-h$ for all $h\in H$. Hence $\alpha_G g=-g$ for 
all $g\in G$.
This implies that
 $G=\mathrm{Ker}(I+\alpha_G)\subset \mathrm{Ker}(I+\alpha)$, i.e., $\lambda$
is supported in $\mathrm{Ker}(I+\alpha)$.
\end{proof}

%\begin{remark}\label{ren1} It is easy to see that Corollary \ref{15.08co1} 
%remains valid if the condition
% \begin{flushleft}
%{\rm(I)} the characteristic functions $\hat\mu_j(y)$  
%do not vanish;
%\end{flushleft}
%is replaced by a weaker condition 
%\begin{flushleft}
%{\rm(II)} the minimal subgroup generated by the sets $\{y\in Y:\hat\mu_j(y)\ne 0\}$, $j=1, 2$,  
%coincides with $Y$.
%\end{flushleft}
%\end{remark}

\begin{corollary}\label{mmco1} 
Consider the 
 group of $p$-adic integers $\mathbb{Z}_p$, where  
$p\ne 2$. 
Let  $\alpha=(c_0, c_1,\dots, c_n,\dots)$, where $c_0\ne p-1$,
 be a topological automorphism of the group $\mathbb{Z}_p$. 
 Let $\xi_1$ and $\xi_2$ be
independent random variables with values in   $\mathbb{Z}_p$ 
and distributions
$\mu_1$ and $\mu_2$. Assume that the conditional  distribution 
of the linear form $L_2 = \xi_1 + \alpha\xi_2$ given 
$L_1 = \xi_1 + \xi_2$ is symmetric. 
Then there is a compact subgroup $G$ of the group $\mathbb{Z}_p$
 such that $\mu_j$ are shifts of $m_G$. Furthermore, if $c_0=1$, then $G=\{0\}$, i.e. 
$\mu_j$ are degenerate distributions.
 \end{corollary}
\begin{proof}
The group $\mathbb{Z}_p$ is a group of the form (\ref{e20.35}). 
Therefore, Theorem \ref{nth1} is applicable to the group $\mathbb{Z}_p$.
It follows from $c_0\ne p-1$ that $I+\alpha\in\mathrm{Aut}(\mathbb{Z}_p)$.
Since any subgroup of the group $\mathbb{Z}_p$ is characteristic,
this implies that $(I+\alpha)(G)=G$. Hence $m_{(I+\alpha)(G)}=m_G$. 
By statement (iii) of Theorem \ref{nth1},
the Haar distribution $m_{(I+\alpha)(G)}$ is a factor of $\lambda$.
Inasmuch as $\lambda$ is supported in $G$, 
it follows from this that $\lambda=m_G$.

By Theorem \ref{nth1}, 
we have $(I-\alpha)(G)=G$. If $c_0=1$, then the only subgroup $G$ 
of the group $\mathbb{Z}_p$, satisfying the condition $(I-\alpha)(G)=G$, is $G=\{0\}$. 
\end{proof}
We note that Corollary \ref{mmco1}
was proved in a different way in \cite[Theorem 2]{Fe2015a},  
see also \cite[Theorem 13.3]{Febooknew}.

\begin{remark}\label{re1}Now we will show that Theorem  \ref{nth1}  
 can not be 
strengthened by narrowing the class of distributions which are 
characterized by the symmetry of the conditional distribution of 
one linear form given another.  
 
Let $X$ be a second countable locally compact Abelian group and let 
$\alpha$ be a topological automorphism of the group $X$. Let $G$ be a 
closed subgroup 
of $X$ such that $\alpha(G)=G$, i.e. $\alpha$ is a topological
automorphism of the group $G$. Suppose
that $I-\alpha_G\in\mathrm{Aut}(G)$ and $\overline{(I+\alpha_G)(G)}$ 
is a compact subgroup
of $G$. We note that $(I+\alpha)(G)=(I+\alpha_G)(G)$.
 Let $\lambda$ be a  distribution supported in  $G$ such that 
 $G$ is the minimal subgroup containing the support of 
$\lambda$. 
Assume also that the Haar 
 distribution $m_{\overline{(I+\alpha_G)(G)}}$ is a factor of  $\lambda$. 
 Let $x_1$ and $x_2$ be elements of  the group 
 $X$ such that 
 \begin{equation}\label{nz}
x_1+\alpha x_2=0.
\end{equation}
 Put $\mu_j=\lambda*E_{x_j}$, where $E_{x_j}$ is the degenerate distribution
concentrated at the element $x_j$, $j=1, 2$. 
 
 Let $\xi_j$   be
independent random variables with values in the group $X$ and distributions
$\mu_j$. We will verify that the conditional  distribution of the 
linear form $L_2 = \xi_1 + \alpha\xi_2$ given $L_1 = \xi_1 + \xi_2$ is symmetric.

Denote by $Y$ the character group of the group $X$ and 
by $H$ the character group of the group $G$. Note that 
$\mathrm{Ker}(I+\widetilde\alpha_G)=A(H, \overline{(I+\alpha_G)(G)})$. 
Taking into account
(\ref{fe22.1}) and the fact that  the Haar distribution 
$m_{\overline{(I+\alpha_G)(G)}}$ is a factor of  $\lambda$, 
the characteristic function
$\hat\lambda(h)$ is represented in the form (\ref{z}), where $a(h)$ 
is a characteristic function on $\mathrm{Ker}(I+\widetilde\alpha_G)$. 
We will check that 
$\hat\lambda(h)$ satisfies the equation
\begin{equation}\label{new11.04.1}
\hat\lambda (u+v )\hat\lambda (u+\widetilde\alpha_G v )=
\hat\lambda (u-v )\hat\lambda (u-\widetilde\alpha_G v), \quad u, v \in H.
\end{equation}
Consider 3 cases.

1. $u, v\in \mathrm{Ker}(I+\widetilde\alpha_G)$. Since 
$\widetilde\alpha_G h=-h$ for all
$h\in \mathrm{Ker}(I+\widetilde\alpha_G)$,  equation (\ref{new11.04.1}) 
becomes an equality.

2. Either $u\in \mathrm{Ker}(I+\widetilde\alpha_G)$, 
$v\notin \mathrm{Ker}(I+\widetilde\alpha_G)$
or $u\notin \mathrm{Ker}(I+\widetilde\alpha_G)$, 
$v\in \mathrm{Ker}(I+\widetilde\alpha_G)$. Then
$u\pm v\notin \mathrm{Ker}(I+\widetilde\alpha_G)$ and both sides of 
equation  (\ref{new11.04.1}) are equal to zero.

3. $u, v\notin \mathrm{Ker}(I+\widetilde\alpha_G)$. Assume that the left-hand side of
equation 
(\ref{new11.04.1}) is not equal to zero. Then 
$u+v\in \mathrm{Ker}(I+\widetilde\alpha_G)$ 
and
 $u+\widetilde\alpha_G v\in \mathrm{Ker}(I+\widetilde\alpha_G)$. This implies 
 that $(I-\widetilde\alpha_G)v\in \mathrm{Ker}(I+\widetilde\alpha_G)$. Inasmuch as  
  $I-\alpha_G\in\mathrm{Aut}(G)$, we have
  $I-\widetilde\alpha_G\in\mathrm{Aut}(H)$ and 
  $(I-\widetilde\alpha_G)v\in \mathrm{Ker}(I+\widetilde\alpha_G)$ 
  implies that $v\in  \mathrm{Ker}(I+\widetilde\alpha_G)$ that contradicts 
the assumption. Thus,  the left-hand side of equation 
(\ref{new11.04.1}) is equal to zero.  Arguing similarly, we make sure that if
$u, v\notin \mathrm{Ker}(I+\widetilde\alpha_G)$, then the right-hand side of
equation (\ref{new11.04.1}) is also equal to zero. 
Thus, both sides of equation  (\ref{new11.04.1}) are equal to zero. 

Put $\lambda_1=\lambda_2=\lambda$. The characteristic function
$\hat\lambda(h)$ satisfies  equation (\ref{new11.04.1}) on the group $H$.
Consider $\lambda$ 
as a distribution on the group $X$. Then the characteristic functions  
 $\hat\lambda_j(y)$  satisfy  equation (\ref{11.04.1}) on the group $Y$. 
It follows from (\ref{nz}) that 
$$
(x_1, u+v )(x_2, u+\widetilde\alpha v )=(x_1, u-v )(x_2, u-\widetilde\alpha v ), 
\quad u, v\in Y.
$$
Hence  the characteristic functions $\hat\mu_j(y)=\hat\lambda(y)(x_j, y)$ 
also satisfy equation (\ref{11.04.1}). 
By Lemma  \ref{lem1},  this implies that  the conditional  distribution of the 
linear form $L_2 = \xi_1 + \alpha\xi_2$ given $L_1 = \xi_1 + \xi_2$ is symmetric.

From the above it follows that $(I-\alpha)(G)=G$  and 
statements (i)--(iv) of Theorem \ref{nth1} 
are true for the subgroup $G$, the distributions $\mu_1$, $\mu_2$, 
and $\lambda$.
\end{remark}

\begin{remark} The main part of Theorem \ref{nth1} is the statement  
that the distributions $\mu_j$ are shifts of the same distribution 
$\lambda$ (statement 1(i)).
We will show that statement 1(i) of Theorem \ref{nth1}  is not true for an 
arbitrary compact totally disconnected Abelian group. To prove this we  use 
the following example constructed in \cite[Remark 2]{FeTVP}.

Let $G$ be a compact Abelian group.
Consider the direct product $$X=\mathop \mathbf{P}
\limits_{j=-\infty}^\infty G_j,$$ where $G_j=G$, $j=0, \pm1, \pm2, \dots$ 
If the group $G$ is totally disconnected, so is $X$. 
Denote by $(g_j)_{j=-\infty}^{\infty}$, where $g_j\in G$, 
elements of the group $X$.
Let $\alpha$ be a topological automorphism of $X$ of the form
 $$\alpha(g_j)_{j=-\infty}^{\infty}=
(g_{j-2})_{j=-\infty}^{\infty}, \quad
(g_j)_{j=-\infty}^{\infty} \in X.$$
Consider the subgroups of $X$ of the form
$$
K_1=\{(g_j)_{j=-\infty}^{\infty}\in X:g_1=0\}, \quad 
K_2=\{(g_j)_{j=-\infty}^{\infty}\in X:g_2=0\} 
$$
and the Haar distributions  $m_{K_1}$ and $m_{K_2}$.
It is obvious that there is no distribution $\lambda$ such that 
$m_{K_j}$ are shifts of $\lambda$.
Let $\xi_1$ and
$\xi_2$ be independent random variables with values in $X$ and
distributions  $m_{K_1}$ and $m_{K_2}$.  As has been proven in
\cite[Remark 2]{FeTVP}, the conditional  distribution 
of the linear form $L_2 = \xi_1 + \alpha\xi_2$ given 
$L_1 = \xi_1 + \xi_2$ is symmetric.

\end{remark}

\section{ Heyde theorem for $p$-quasicyclic groups}

We prove in this section an analogue of Heyde's 
theorem for $p$-quasicyclic groups $\mathbb{Z}(p^\infty)$, where $p\ne 2$.

Let $Y$ be an Abelian group, let $f(y)$ be a function on  $Y$, and let $h$ be an
element of $Y$. Denote by $\Delta_h$ the finite difference operator 
$$\Delta_h f(y)=f(y+h)-f(y), \quad y\in Y.$$ A function $f(y)$ on the group 
$Y$ is called
a  polynomial  if
$$\Delta_{h}^{n+1}f(y)=0$$ for some nonnegative integer 
$n$ and all $y, h \in Y$.

\begin{theorem}\label{nth2}  Consider the 
$p$-quasicyclic group $\mathbb{Z}(p^\infty)$, where  
$p\ne 2$. 
Let  $\alpha$ be an automorphism of the group $\mathbb{Z}(p^\infty)$. 
 Let $\xi_1$ and $\xi_2$ be
independent random variables with values in   $\mathbb{Z}(p^\infty)$ 
and distributions
$\mu_1$ and $\mu_2$. Assume that the conditional  distribution of 
the linear form $L_2 = \xi_1 + \alpha\xi_2$ given 
$L_1 = \xi_1 + \xi_2$ is symmetric. 
If $\alpha\ne-I$, then there is a finite subgroup $G$ of the group 
$\mathbb{Z}(p^\infty)$
satisfying the condition $(I-\alpha)(G)=G$  and a distribution  
$\lambda$ supported in $G$ such that 
statements {\rm (i)--(iv)} of Theorem {\rm\ref{nth1}} are true. 
If $\alpha=-I$, then $\mu_1=\mu_2$.
\end{theorem}

To prove   Theorem \ref{nth2}  we need two lemmas.
\begin{lemma}[{\!\!\protect\cite[Lemma 9.17]{Febooknew}}] 
\label{newle18.01.2}  
 Let $Y$
be an Abelian group  and let $\beta$ be an automorphism of $Y$.
Assume that the functions $\varphi_j(y)$ satisfy the equation 
\begin{equation}\label{e20.30}
\varphi_1(u+v)+\varphi_2(u+\beta v)-\varphi_1(u-v)-
\varphi_2(u-\beta v)=0,\quad u, v \in
 Y.
\end{equation}
 Then  the function  $\varphi_1(y)$ satisfies the equation 
 \begin{equation}\label{e20.31}
\Delta_{(I-\beta)k_3}\Delta_{2 k_2}\Delta_{(I+\beta)k_1}\varphi_1(y)
= 0,\quad  y\in Y,
\end{equation}
 and the function  $\varphi_2(y)$ satisfies the equation 
\begin{equation}\label{e20.32}
\Delta_{-(I-\beta)k_3}\Delta_{(I+\beta)k_2}\Delta_{2 
\beta k_1}\varphi_2(y) = 0,
 \quad y\in Y,
\end{equation}
 where $k_j$, $j=1, 2, 3$, are arbitrary elements of the group $Y$. 
\end{lemma}
The following lemma is well known. For the proof see, e.g., \cite[Proposition 1.30]{Febooknew}.
\begin{lemma}
\label{s5.6} 
	Let $Y$
	be a  compact Abelian group   and let $f(y)$ be a continuous
	polynomial on $Y$. Then $f(y)=\text{const}$
		for all $y\in Y$.
\end{lemma}
\noindent\textit{Proof of Theorem} \ref{nth2}.
Since the character group of the group $\mathbb{Z}(p^\infty)$ is topologically 
isomorphic to the group of $p$-adic integers
$\mathbb{Z}_p$, to avoid introducing additional notation,  we 
assume that the character group of $\mathbb{Z}(p^\infty)$ is $\mathbb{Z}_p$.
 By Lemma \ref{lem1},
the characteristic functions $\hat\mu_j(y)$   satisfy equation (\ref{11.04.1}).

Put   $\nu_j=\mu_j*\bar\mu_j$, $j=1, 2$. Then 
$\hat\nu_j(y)=|\hat\mu_j(y)|^2\ge 0$  for all $y\in \mathbb{Z}_p$ 
and the characteristic functions   $\hat\nu_j(y)$ also satisfy equation   
(\ref{11.04.1}). Taking into account that   the family of the subgroups
 $\{p^k\mathbb{Z}_p\}_{k=0}^{\infty}$ forms an open 
 basis at the zero of
 the group $\mathbb{Z}_p$,  we can choose a nonnegative integer 
 $l$ in such a way that 
$\hat\nu_j(y)>0$ for all
$y\in p^l\mathbb{Z}_p$, $j=1, 2$. Put $\varphi_j(y)=\ln\hat\nu_j(y)$, 
$y\in p^l\mathbb{Z}_p$.
Inasmuch as the characteristic functions   $\hat\nu_j(y)$ satisfy equation   
(\ref{11.04.1}), the functions $\varphi_j(y)$ satisfy equation 
(\ref{e20.30}), where   $Y=p^l\mathbb{Z}_p$ and
$\beta=\widetilde\alpha$.  By Lemma \ref{newle18.01.2},  the function 
$\varphi_1(y)$ satisfies equation (\ref{e20.31}) and
the function 
$\varphi_2(y)$ satisfies equation (\ref{e20.32}).

Assume that $\alpha$ corresponds to an element 
$c=(c_0, c_1, \dots, c_n,\dots)
\in\mathbb{Z}_p^\times$.
It is obvious that $I$ corresponds to the element $e=(1, 0, \dots, 0,\dots)$.

Let $\alpha\ne\pm I$. 
We have 
 $(e+c)\mathbb{Z}_p=p^a\mathbb{Z}_p$, 
$(e-c)\mathbb{Z}_p=p^b\mathbb{Z}_p$, where
$a$, $b$ are nonnegative integers. Set $m=l+\max\{a, b\}$.  
The multiplication by 2 is a topological automorphism of 
the group $\mathbb{Z}_p$ because $p\ne 2$.
  It follows from (\ref{e20.31})  
and (\ref{e20.32}) that  each of the functions 
$\varphi_j(y)$ on the subgroup $p^m\mathbb{Z}_p$ satisfies the equation
$$
\Delta_h^3\varphi_j(y) = 0,
\quad y, h\in p^m\mathbb{Z}_p, \ j=1, 2.
$$
Thus,  the functions $\varphi_j(y)$ are continuous polynomials on the group 
$p^m\mathbb{Z}_p$. Since $p^m\mathbb{Z}_p$ is a compact Abelian group, by Lemma \ref{s5.6},
the functions $\varphi_j(y)$ are constants. As far as  $\varphi_1(0)=\varphi_2(0)=0$, we have
$\varphi_1(y)=\varphi_2(y)=0$ for all $y\in p^m\mathbb{Z}_p$. This implies that 
$|\hat\mu_1(y)|=|\hat\mu_2(y)|=1$ for all $y\in p^m\mathbb{Z}_p$.
Put   $K=A(\mathbb{Z}(p^\infty), p^m\mathbb{Z}_p)$. 
By Lemma \ref{lem11},  we can  replace  the distributions  $\mu_j$ by 
their shifts  
$\rho_j$   in such a way that   $\rho_j$ are supported in $K$,
and   if   
$\zeta_j$
are independent random variables with values in
the group $K$  and distributions $\rho_j$, then the conditional distribution 
of the linear form $N_2=\zeta_1 + \alpha_K\zeta_2$ given $N_1=\zeta_1 + \zeta_2$  
is symmetric. It is easy to see that   $K=\mathbb{Z}(p^m)$  
and hence $K$ is a group of the form (\ref{e20.35}). 
 The group $K$, the automorphism $\alpha_K$, 
the independent random variables $\zeta_j$, and the 
distributions $\rho_j$
 satisfy all conditions of Theorem \ref{nth1}.
The assertion of the theorem follows from Theorem \ref{nth1}. 

Let $\alpha=I$. Substituting $u=v=y$ in equation (\ref{11.04.1}), we get
$\hat\mu_1(2y)\hat\mu_2(2y)=1$ for all $y\in \mathbb{Z}_p$. Since 
 the multiplication by 2 is a topological automorphism of 
the group $\mathbb{Z}_p$, this implies that 
$|\hat\mu_1(y)|=|\hat\mu_2(y)|=1$ for all $y\in \mathbb{Z}_p$. 
Hence $\mu_j$ are degenerate
distributions and $G=\{0\}$. 

Let $\alpha=-I$. Substituting $u=v=y$ in equation (\ref{11.04.1}) we get
$\hat\mu_1(2y)=\hat\mu_2(2y)$  and hence 
$\hat\mu_1(y)=\hat\mu_2(y)$ for all $y\in \mathbb{Z}_p$. This implies that
$\mu_1=\mu_2$. 

We note that the reasoning given in Remark \ref{re1} shows that Theorem  \ref{nth2} can not be 
strengthened by narrowing the class of distributions which are 
characterized by the symmetry of the conditional distribution of 
one linear form given another. 
$\hfill\Box$
 
\begin{corollary}\label{mco1} 
Consider the 
$p$-quasicyclic group $\mathbb{Z}(p^\infty)$, where  
$p\ne 2$.  
Let  $\alpha=(c_0, c_1,\dots, c_n,\dots)$, where $c_0\ne p-1$,
 be an   automorphism of the group $\mathbb{Z}(p^\infty)$. 
 Let $\xi_1$ and $\xi_2$ be
independent random variables with values in   $\mathbb{Z}(p^\infty)$ 
and distributions
$\mu_1$ and $\mu_2$. Assume that the conditional  distribution 
of the linear form $L_2 = \xi_1 + \alpha\xi_2$ given 
$L_1 = \xi_1 + \xi_2$ is symmetric. 
Then there is a finite subgroup $G$ of the group $\mathbb{Z}(p^\infty)$
 such that $\mu_j$ are shifts of $m_G$. Furthermore, if $c_0=1$, then $G=\{0\}$, i.e. 
$\mu_j$ are degenerate distributions.
 \end{corollary}
 
 \begin{proof} The proof repeats the proof of Corollary \ref{mmco1}, 
 where instead of Theorem \ref{nth1} we use Theorem \ref{nth2}.
\end{proof} 

Let ${\mathcal P}$ be a set of pairwise distinct prime numbers such that 
$2\notin {\mathcal P}$. Based on the proofs of Theorems \ref{nth1} and \ref{nth2}, we prove 
in the final part of the article a similar statement for the  direct 
product of a group of the form (\ref{e20.35})
 and the $p$-quasicyclic group $\mathbb{Z}(p^\infty)$, 
provided that $p\ne 2$ and $p\notin {\mathcal P}$.

Let $X=K\times\mathbb{Z}(p^\infty)$, where $K$ is a group of the form 
(\ref{e20.35}) and $p\notin {\mathcal P}$.
Denote by
$(k, a)$, where $k\in K$, $a\in\mathbb{Z}(p^\infty)$, 
elements
of the group  $X$. Let $\alpha\in\mathrm{Aut}(X)$. It is obvious that
 $\alpha$ acts as follows:   $\alpha(k, a)=(\alpha_{K}k, \alpha_{\mathbb{Z}(p^\infty)}a)$. 
 We will write $\alpha$ in the form $\alpha=(\alpha_{K}, 
 \alpha_{\mathbb{Z}(p^\infty)})$.
Denote by $Y$ the character group of the group $X$ and by $L$ the character group of 
the group $K$. The group $Y$ is 
topologically isomorphic to the group $L\times\mathbb{Z}_p$.
Denote by $(l, b)$, where $l\in L$, $b\in\mathbb{Z}_p$, elements
of the group  $Y$. Then 
$\widetilde\alpha(l, b)=(\widetilde\alpha_Kl, \widetilde\alpha_{\mathbb{Z}(p^\infty)}b)$.
We will write $\widetilde\alpha$ in the form $\widetilde\alpha=(\widetilde\alpha_{K}, 
 \widetilde\alpha_{\mathbb{Z}(p^\infty)})$.
\begin{theorem}\label{nth3}  Let  $X=K\times\mathbb{Z}(p^\infty)$, 
where $K$ is a group of the form $(\ref{e20.35})$, $p\ne 2$, and $p\notin {\mathcal P}$. 
Let  $\alpha=(\alpha_{K}, \alpha_{\mathbb{Z}(p^\infty)})$ 
be a topological automorphism of the group $X$.  Let $\xi_1$ and $\xi_2$ be
independent random variables with values in   $X$ 
and distributions
$\mu_1$ and $\mu_2$. Assume that the conditional  distribution of 
the linear form $L_2 = \xi_1 + \alpha\xi_2$ given 
$L_1 = \xi_1 + \xi_2$ is symmetric. 
Then there is a closed subgroup $G$ of the group $X$ satisfying the 
conditions: 
\renewcommand{\labelenumi}{\rm(\Roman{enumi})}
\begin{enumerate}
\item
$(I-\alpha)(G)=G$;
	\item
$(I+\alpha)(G)$ is a compact subgroup, 
\end{enumerate} 
and a distribution $\lambda$
supported in $G$ 
such that statements {\rm (i)--(iv)} of Theorem {\rm\ref{nth1}} are true. 
Moreover, if $\alpha_{\mathbb{Z}(p^\infty)}\ne - I$, then the subgroup 
$G$ is compact, and if
$\alpha_{\mathbb{Z}(p^\infty)}=-I$, then either $G$ is   compact  
or 
$G=M\times\mathbb{Z}(p^\infty)$, 
where $M$ is a compact subgroup of  $K$.
\end{theorem}
\begin{proof}
Assume that $\alpha_{\mathbb{Z}(p^\infty)}\ne \pm I$.   Arguing as in 
the proof of Theorem \ref{nth2} in the case when $\alpha\ne \pm I$, 
we prove that there is a nonnegative integer $m$ such that 
$|\hat\mu_1(0, b)|=|\hat\mu_2(0, b)|=1$ for all $b\in p^m\mathbb{Z}_p$.
Applying Lemma \ref{lem11}, we reduce the
proof of the theorem to the case when independent random variables 
take values
in  the group
$A(K\times\mathbb{Z}(p^\infty), p^m\mathbb{Z}_p)=
K\times\mathbb{Z}(p^m)$. Since $p\ne 2$ and $p\notin {\mathcal P}$, 
the group $K\times\mathbb{Z}(p^m)$ is  
a group of the form (\ref{e20.35}).  
The statement of the theorem follows from Theorem \ref{nth1}
applied to the group $K\times\mathbb{Z}(p^m)$.

Suppose that $\alpha_{\mathbb{Z}(p^\infty)}=I$. By Lemma \ref{lem1}, 
the characteristic functions $\hat\mu_j(l, b)$   
satisfy equation (\ref{11.04.1})
which takes the form
\begin{multline}\label{29nn11.04.1}
\hat\mu_1(l_1+l_2, b_1+b_2)\hat\mu_2(l_1+\widetilde\alpha_{\mathbb{Z}(n)}l_2, 
b_1+b_2)\\=
\hat\mu_1(l_1-l_2, b_1-b_2)\hat\mu_2(l_1-\widetilde\alpha_{\mathbb{Z}(n)}l_2, 
b_1-b_2), \quad l_j\in L, \ b_j \in \mathbb{Z}_p.
\end{multline} 
Putting $l_1=l_2=0$, $b_1=b_2=b$ in equation (\ref{29nn11.04.1}), we get
\begin{equation}\label{29nnn11.04.1}
\hat\mu_1(0, 2b)\hat\mu_2(0, 2b)=1, \quad 
 b  \in \mathbb{Z}_p.
\end{equation}
Inasmuch as the multiplication by 2 is a topological automorphism of 
the group $\mathbb{Z}_p$, we obtain from (\ref{29nnn11.04.1}) that 
\begin{equation}\label{30nnn11.04.1}
|\hat\mu_1(0, b)|=|\hat\mu_2(0, b)|=1, \quad 
 b  \in \mathbb{Z}_p.
\end{equation}
In view of (\ref{30nnn11.04.1}), we can apply Lemma  \ref{lem11} and reduce
 the proof of the theorem to the case when independent random variables take 
 values
in  the group  
$A(K\times\mathbb{Z}(p^\infty), \mathbb{Z}_p)=K$. 
The statement of the theorem follows from Theorem \ref{nth1}
applied to the group $K$.

Thus, if $\alpha_{\mathbb{Z}(p^\infty)}\ne -I$, then the theorem is proved.
In this case the subgroup $G$ is compact.

It remains to prove the theorem in the case when $\alpha_{\mathbb{Z}(p^\infty)}=-I$.
Put 
$$
S=\{(l, b)\in L\times\mathbb{Z}_p:|\hat\mu_1(l, b)|=
|\hat\mu_2(l, b)|=1\}, \quad G=A(K\times\mathbb{Z}(p^\infty), S).
$$
We note that  the multiplication by 2 is a topological automorphism of 
any closed subgroup  of the groups $K\times\mathbb{Z}(p^\infty)$ 
and $L\times\mathbb{Z}_p$. 
 This implies that $S^{(2)}=S$. Applying Lemma \ref{lem11},  we 
replace  the distributions  $\mu_j$ by their shifts  
$\lambda_j$   in such a way that  $\lambda_j$ are supported in $G$, 
and   if   $\eta_j$
are independent random variables with values in
the group $G$  and distributions $\lambda_j$, then the conditional 
distribution of the linear form $M_2=\eta_1 + \alpha_G\eta_2$ given 
$M_1=\eta_1 + \eta_2$  is symmetric. Denote by $H$ the character 
group of the group $G$. Considering $\lambda_j$ as distributions on $G$,
we have
\begin{equation}\label{16.08.1}
\{h\in H:|\hat\lambda_1(h)|=|\hat\lambda_2(h)|=1\}=\{0\}.
\end{equation}
If $G$ is a compact group, then $G$ 
is topologically isomorphic to
a group of the form (\ref{e20.35}). In view of  (\ref{16.08.1}),  
the statement of the theorem follows from Corollary \ref{nnco1} applied to 
the group $G$. 

Assume that $G$ is not a compact group. Then $G=M\times\mathbb{Z}(p^\infty)$, 
where $M$ is a compact subgroup of  $K$. 
Next, we follow the scheme of the proof of Theorem \ref{nth1}.
Denote by $N$ the character group of 
the group $M$. Then $H$ is topologically isomorphic to the group
$N\times\mathbb{Z}_p$.  To avoid introducing additional notation, we  
 suppose that $H=N\times\mathbb{Z}_p$.  
 Denote by $(n, b)$, where $n\in N$, $b\in\mathbb{Z}_p$, elements
of the group  $H$.
 Putting $T=\mathrm{Ker}(I-\widetilde\alpha_G)$ 
 and arguing as in item 1 of the proof of Theorem \ref{nth1}, 
 we get that $T=\{0\}$. This implies that $(I-\alpha)(G)=G$.
Since $\alpha_G=(\alpha_M, -I)$ and $M$ is a group of the form (\ref{e20.35}), 
it follows from 
 $(I-\alpha_G)(G)=(I-\alpha)(G)=G$  that 
$I-\alpha_G\in\mathrm{Aut}(G)$ and hence
$I-\widetilde\alpha_G\in \mathrm{Aut}(H)$.

  Put   $\nu_j=\lambda_j*\bar\lambda_j$, $j=1, 2$, and $f(h)=\hat\nu_1(h)$, 
$g(h)=\hat\nu_2(h)$, $h\in H$. By Lemma \ref{lem1}, 
the characteristic functions $f(h)$   and $g(h)$
satisfy equation (\ref{new1}).
Inasmuch as $\alpha_G=(\alpha_M, -I)$, we have 
$\widetilde\alpha_G=(\widetilde\alpha_M, -I)$ and hence
$$
\kappa=-f_4\widetilde\alpha_G(I-\widetilde\alpha_G)^{-2}=
(-f_4\widetilde\alpha_M(I-\widetilde\alpha_M)^{-2}, I).
$$  
Since the group $N$ is topologically isomorphic to a group
of the form (\ref{1e20.35}), this implies that $\kappa^mh=h$ 
for any $h\in H$ and 
some natural $m$, where $m$
  depends, generally speaking, on $h$.
Since $f_2\in \mathrm{Aut}(H)$, all conditions of Lemma \ref{newle1},   
where $Y=H$, 
 $\beta=\widetilde\alpha_G$, and $y_0$ is an arbitrary element of
 $H$, are fulfilled. By Lemma \ref{newle1}, the 
functions $f(h)$ and $g(h)$ satisfy  equations
(\ref{11.04.16}) and (\ref{11.04.8}) and equalities 
(\ref{11.04.14}) and (\ref{11.04.15})
are true for all $h\in H$.    

We note that $\mathrm{Ker}(I+\widetilde\alpha_G)=
\mathrm{Ker}(I+\widetilde\alpha_M)\times\mathbb{Z}_p$ and prove that
if either $f(h)\ne 0$ or $g(h)\ne 0$, then  
$h\in \mathrm{Ker}(I+\widetilde\alpha_G)$. 
Suppose that $f(h_1)\ne 0$ at  an element $h_1=(n_1, b_1)\in H$. 
Consider the subgroup $\langle  n_1 \rangle$ of the group $N$ generated by
the element $n_1$. Put 
$$
h_2=-2\widetilde\alpha_G(I-\widetilde\alpha_G)^{-1}h_1=
(-2\widetilde\alpha_M(I-\widetilde\alpha_M)^{-1} n_1, b_1)=(n_2, b_1),
$$
where $n_2=-2\widetilde\alpha_M(I-\widetilde\alpha_M)^{-1} n_1$. We have
$-f_2\widetilde\alpha_M(I-\widetilde\alpha_M)^{-1}\in\mathrm{Aut}(N)$. 
This implies that $n_2\in \langle n_1\rangle$. Moreover, the elements $n_1$ 
and $n_2$ have the same order  and hence $\langle n_1\rangle=\langle n_2\rangle$. 
It follows from (\ref{11.04.14}) that $g(h_2)\ne 0$. 
Put 
\begin{equation}\label{1nn11.04.15}
u_1=-(I+\widetilde\alpha_G)(I-\widetilde\alpha_G)^{-1} h_1=
(-(I+\widetilde\alpha_M)(I-\widetilde\alpha_M)^{-1}n_1, 0).
\end{equation}
We find from   (\ref{11.04.16}) and (\ref{11.04.14}) that  then $f(u_1)=1$.  
By Lemma \ref{lem2}, the set $\{h\in H: f(h)=1\}$ is a subgroup. 
As far as $f(u_1)=1$, we conclude that
$f(h)=1$ at each element  of the subgroup $\langle u_1\rangle$.
Put
\begin{equation}\label{1nnn11.04.15}
u_2=(I+\widetilde\alpha_G)(I-\widetilde\alpha_G)^{-1} 
h_2=((I+\widetilde\alpha_M)(I-\widetilde\alpha_M)^{-1}n_2, 0).
\end{equation}
Since $g(h_2)\ne 0$, we find from  (\ref{11.04.8}) and  (\ref{11.04.15})   
that  
$g(u_2)=1$. Taking into account that the set $\{h\in H: g(h)=1\}$ 
is a subgroup, 
we see that
$g(h)=1$ at each element  of the subgroup $\langle u_2\rangle$.
As far as the elements $n_1$ and $n_2$ have the same order, 
it follows from (\ref{1nn11.04.15}) and  (\ref{1nnn11.04.15}) that the elements 
$u_1$ and $u_2$ also have the same order. 
Hence
$\langle u_1\rangle=\langle u_2\rangle$. 
Thus,  we proved that $f(h)=g(h)=1$ at each element of the subgroup
$\langle u_1\rangle$. 
It follows from (\ref{16.08.1}) that   if
$f(h)=g(h)=1$ at an element $h\in H$, then $h=0$. Inasmuch as
$f(h)=g(h)=1$ at each element of the subgroup
$\langle u_1\rangle$, we conclude that $\langle u_1\rangle=\{0\}$, i.e., 
$u_1=0$. Taking into account (\ref{1nn11.04.15}) and the fact that 
$(I-\widetilde\alpha_M)^{-1}\in\mathrm{Aut}(N)$, we see that  
 if $f(h_1)\ne 0$ at  an element $h_1=(n_1, b_1)\in H$, then 
 $n_1\in \mathrm{Ker}(I+\widetilde\alpha_M)$. Taking into account that $\widetilde\alpha_G=(\widetilde\alpha_M, -I)$,
 this implies that $h_1\in \mathrm{Ker}(I+\widetilde\alpha_G)$.
 
For the function $g(h)$ we  argue  similarly.
 The final part of the proof repeats verbatim the reasoning in items 4 and 5 
of the proof of Theorem \ref{nth1}.
\end{proof}

\medskip

\noindent\textbf{\Large Acknowledgements}

\medskip

\noindent I  thank the referee  for  a careful reading of the manuscript and
  useful remarks.
  
\medskip

\noindent\textbf{Funding} The author has not disclosed any funding.
 
\medskip

\noindent\textbf{Data Availability Statement} Data sharing 
not applicable to this article as no datasets were
generated or analysed during the current study.

\medskip
 
\noindent\textbf{Declarations}   

\medskip

\noindent\textbf{Conflict of interest} The author states that there is no 
conflict of interest.

\end{document}